\documentclass[12pt]{article}
\usepackage{amsmath,dsfont}
\usepackage{amsthm}
\usepackage{enumerate}
\usepackage{amssymb}
\usepackage{multirow}
\usepackage[latin1]{inputenc}
\usepackage{color}
\usepackage{float}
\usepackage{graphicx}
\begin{document}
\newcommand{\R}{{\bf R}}
\newcommand{\K}{{\bf K}}
\newcommand{\C}{{\bf C}}
\newcommand{\seq}[1]{(#1_n)_{n=1,2,\dots}}
\newcommand{\cs}[1]{{\color{blue}#1}}
\newtheorem{thm}{Theorem}[section]
\newtheorem{defn}[thm]{Definition}
\newtheorem{cor}[thm]{Corollary}
\newcommand{\rtref}[1]{{\rm \ref{#1}}}
\newcommand{\rmref}[1]{{\rm (\ref{#1})}}
\newcommand{\rmcite}[1]{{\rm \cite{#1}}}
\newcommand{\rmciter}[2]{{\rm \cite[#1]{#2}}}
\newtheorem{lem}[thm]{Lemma}
\newtheorem{prop}[thm]{Proposition}
\newtheorem{rem}[thm]{Remark}
\newtheorem{example}[thm]{Example}

\title {A new characterization of the hazard rate and reversed hazard rate orders with applications}

\author{C. Sangüesa \footnote{e-mail:  csangues@unizar.es}\\ \small Department of Statistical Methods and IUMA, \\ \small University of Zaragoza, Zaragoza, 50009,  SPAIN}

\date{}

\begin{titlepage}
\setcounter{page}{1} \maketitle

\bigskip \bigskip
\begin{abstract}
We propose a general characterization of the hazard rate and reversed hazard rate stochastic orders for random variables that are absolutely continuous with respect to a common dominating measure. This framework is useful in giving a unified treatment of continuous, discrete, and mixed distributions without requiring ad hoc approximation techniques or restrictive integrability assumptions. Using this result, we provide direct proofs of the bivariate characterizations of both orders. Additionally, we introduce the class of $\overline{G}$-IFR and $\overline{G}$-DRHR distributions on additive groups, unifying aging properties across different domains and proving their closure under convolution. Finally, we revisit and extend the classic results of Shanthikumar and Yao (1991) on the preservation of hazard rate orders under random sums, simplifying the underlying conditions and accommodating discrete and mixed sum components.

\end{abstract}

\bigskip \bigskip
2020 Mathematics Subject Classification: 60E15; 62N05; 60G50.

\bigskip \bigskip
{\it \cs{Keywords} and phrases}: Stochastic orders; Hazard rate order; Reversed hazard rate order; Dominating measure; Bivariate characterization; Random sum; IFR distribution

\bigskip \bigskip

\bigskip \noindent

\end{titlepage}

\maketitle

\baselineskip=1.5\baselineskip

\section{Introduction}

Stochastic orderings provide a fundamental framework for comparing random variables in applied probability, reliability theory, queueing models, and risk analysis (see, e.g., Müller and Stoyan \cite{mustst}, Marshall and Olkin \cite{maolli}, and Shaked and Shanthikumar \cite{shshst}). Among these relations, the hazard rate order ($\leq_{\text{hr}}$) and the reversed hazard rate order ($\leq_{\text{rh}}$) play a crucial role when assessing aging properties, repairable systems, and optimal allocation or scheduling policies.  Classical characterizations of $\leq_{\text{hr}}$ and $\leq_{\text{rh}}$ using hazard rate functions or densities are typically presented under two distinct setups: either assuming absolute continuity with respect to the Lebesgue measure (for continuous distributions) or restricting to integer-valued random variables. While intuitive, this separation often complicates technical proofs when dealing with mixed distributions or point masses (such as burn-in effects at the origin or censoring points). Furthermore, bivariate characterizations of these orders -initially introduced by Shanthikumar and Yao \cite{shyabi} for the hazard rate order and later adapted by Kijima and Ohnishi \cite{kionpo} for the reversed hazard rate order- frequently require approximation arguments or additional integrability conditions when extended to general settings (cf. Müller and Stoyan \cite[Ch. 1]{mustst}).  The main purpose of this paper is to establish a rigorous, unified, and self-contained measure-theoretic framework for the hazard rate and reversed hazard rate orders using general densities with respect to a common dominating measure $\mu$. Our approach avoids functional approximations or restrictive moment conditions, providing direct proofs for their bivariate characterizations.  Based on this general formulation, the contributions of this manuscript are threefold: First, we prove necessary and sufficient density-based conditions for $\leq_{\text{hr}}$ and $\leq_{\text{rh}}$ comparisons under an arbitrary dominating measure $\mu$, unifying the continuous, discrete, and mixed cases into a single mathematical formulation (Theorem \ref{tgf1}). Second, we introduce the classes of $\overline{G}$-IFR and $\overline{G}$-DRHR distributions defined on the closure of additive groups $G \subseteq \mathbb{R}$ (Definition \ref{dfgclas}). This formulation embraces standard continuous IFR/DRHR properties ($G = \mathbb{R}$) as well as discrete IFR/DRHR properties ($G = \mathbb{Z}$), proving their closure properties under convolution (Theorem \ref{thclco} and Theorem \ref{thpcon}). Third, we revisit the celebrated result of Shanthikumar and Yao \cite[Theorem 3.11]{shyabi} regarding the preservation of the hazard rate order under random summation. By leveraging our general density characterization, we significantly simplify their core density conditions -reducing them to diagonal evaluations- and extend the results to allow discrete/mixed summands as well as random summation indices with positive mass at zero (Theorem \ref{thgera} and Corollary \ref{coge}).  The organization of the paper is as follows. Section \ref{semain} presents the main characterization theorem using general dominating measures and provides illustrative examples. Section \ref{seage} contains the complete, self-contained proof of the bivariate characterizations. In Section \ref{sepres}, we introduce aging properties on groups and discuss their preservation under summation. Finally, Section \ref{sestoc} addresses stochastic comparisons for random sums and discusses several applications.

\section{Main result} \label{semain}
Let $X$ be a random variable with respective cumulative distribution function (cdf) $F$.  We will say that $F$ is absolutely continuous with respect to a dominating measure $\mu$ if there exists a $\mu$-measurable, non-negative function $f$ (called a density) such that
\[F(x)=\int_{(-\infty,x]}f(u)d\mu(u), \quad x\in\mathbb{R}. \]
Every random variable is absolutely continuous with respect to a dominating measure (for instance, its own probability measure). Moreover, let $Y$ be another random variable with cdf $G$. We will say that $F$ and $G$ are absolutely continuous with respect to the common dominating measure $\mu$ if both $F$ and $G$ are absolutely continuous with respect to $\mu$. It is well known that given two cdfs $F$ and $G$ we can always find a dominating measure for which both $F$ and $G$ are absolutely continuous, for instance
\[\mu((-\infty,x])=F(x)+G(x),\quad x\in \mathbb{R}.\]
From now on we recall that the support of a real measure on $\mathbb{R}$ is defined as
\[Supp(\mu)=\{x:\ \mu((x-\epsilon,x+\epsilon))>0 \quad \hbox{for all $\epsilon>0$}\},\]
and similarly, for a random variable $X$, $supp(X)$ will denote the support of the probability measure generated by $X$. As usual, we will say that $\mu$ has support on a Borel set $A$ if $\mu(A)=1$.  Analogously, we will say that $X$ has support on $A$ if its probability measure does.
It is also well known that a density function $f$ with respect to a dominating measure $\mu$ satisfies (cf. \cite[p.419]{biprob})
\begin{equation}
	f(x)=\lim_{h\downarrow 0}\frac{F(x+h)-F(x-h)}{\mu((x-h,x+h])}, \quad  \mu \hbox{ a.e.} \label{limder}
\end{equation}
Finally, we will assume with no loss of generality that if $x\not \in Supp(\mu)$, $f(x)=0$. Our aim in this paper is to characterize the hr and rh orders using general densities.  First of all, we recall the definitions of the most common stochastic orders to be used along the paper. (cf. \cite{mustst}, \cite{shshst} or \cite{beanin}) 
\begin{defn}
	Let $X$ and $Y$ be random variables, with cdfs $F$ and $G$, and survival functions $\bar{F}:=1-F$ and $\bar{G}:=1-G$, respectively. Then $X$ is said to be smaller than $Y$ in the
	\begin{itemize}
		\item[(a)] usual stochastic order ($X\leq_{\text{st}} Y$) if $\bar{G}(x)\geq \bar{F}(x)$ for all $x \in \mathbb{R}$.
		\item[(b)] hazard rate order ($X\leq_{\text{hr}} Y$) if $\displaystyle \frac{\bar{G}(x)}{\bar{F}(x)}$ is increasing in $x$ (i.e. $\bar{G}(y) \bar{F}(x)-\bar{G}(x) \bar{F}(y)\geq 0, \ x\leq y$).
		\item[(c)] reversed hazard rate order ($X\leq_{\text{rh}} Y$) if $\displaystyle \frac{G(x)}{F(x)}$ is increasing in $x$ (i.e. $G(y) F(x)-G(x) F(y)\geq 0, \  x\leq y$).
			\item[(d)]Likelihood ratio order ($X\leq_{\text{lr}} Y$) if $X$ and $Y$ are absolutely continuous with respect to a common dominating measure $\mu$, with respective densities $f$ and $g$ satisfying that  $\displaystyle \frac{g(x)}{f(x)}$ is increasing in $x$, $\mu$ a.e. (i.e. $g(y) f(x)-g(x) f(y)\geq 0, \  x\leq y \ (\mu \times \mu)\text{ a.e. } (x,y)$.
	\end{itemize}
	\label{dfstor}
\end{defn}
It is well known that (d) $\Rightarrow$ (a), (b) and (c). Moreover, (b) or (c) $\Rightarrow$ (a).  The next result provides the characterization of the hr and rh orders using densities, as in the definition of the lr order.    

\begin{thm} Let $X$ and $Y$ be random variables absolutely continuous with respect to the same dominating measure $\mu$, with respective densities $f$ and $g$. We have  
	\begin{enumerate}[(a)]
	\item $X\leq_{\text{hr}} Y \ \Longleftrightarrow \displaystyle f(x)\bar{G}(x)\geq g(x)\bar{F}(x)$ $\mu$ a.e.
		\item $X\leq_{\text{rh}} Y \ \Longleftrightarrow \displaystyle f(x)G(x)\leq g(x)F(x)$ $\mu$ a.e.
\end{enumerate} \label{tgf1}
\end{thm}
\begin{proof} Firstly, we show part (a). For the only if case, let us assume that $X\leq_{\text{hr}} Y$. Thus, for any $h>0$ and $x\in \mathbb{R}$ we can write
 \[\bar{G}(x-h)\bar{F}(x+h)\leq \bar{F}(x-h)\bar{G}(x+h).\]
 Observe that this inequality is equivalent to write \begin{equation}(\bar{F}(x-h)-\bar{F}(x+h))\bar{G}(x+h)\geq (\bar{G}(x-h)-\bar{G}(x+h))\bar{F}(x+h). \label{thgeh1}
 	\end{equation}
 The conclusion follows, for $x\in Supp(\mu)$ dividing (\ref{thgeh1}) by $\mu((x-h,x+h])$, taking limits as $h\downarrow 0$ and using (\ref{limder}).  Notice that the set of points outside the support of $\mu$ have measure $0$.
For the if part, we need to prove that $X\leq_{\text{hr}}Y$. We will use the following characterization of the hr order, using the residual lifetimes see \cite[p.8]{mustst}
\begin{equation}
X\leq_{\text{hr}}Y \ \Longleftrightarrow X-t|X> t\leq_{\text{st}} Y-t|Y>t \hbox{ for all $t$ such that $\bar{F}(t)>0$.}
\label{rhres}\end{equation}
To check (\ref{rhres}), we take into account that for a cdf $F$, being absolutely continuous with respect a dominating measure $\mu$, the cumulative hazard function is defined as
\[\Lambda(x)=\int_{(-\infty,x]}\frac{dF(u)}{1-F(u^{-})}=\int_{(-\infty,x]}\frac{f(u)}{1-F(u^{-})}d\mu(u),\quad x\in \mathbb{R}.\]
where, as usual, the integrand takes the value $0$ if $1-F(u^{-})=0$.  In classical survival analysis, $\Lambda$ is usually defined as an integral on the interval $[0,x]$, because non-negative random variables are considered, but there is no problem in extending this integral to $x\in\mathbb{R}$. See, for instance \cite{jaknin}.   Note that, for a non-negative random variable $X$, with cdf $F$, its survival function is uniquely defined by $\Lambda$ using the following formula  (see, for instance, \cite[p.91]{anstat})
\begin{equation}\bar{F}(x)=exp\{-\Lambda^c(x)\}\prod_{\substack{u\leq x,\\ u\in D}}(1-\Delta\Lambda(u)), \quad x\geq 0,\label {surcum}\end{equation}
where $D$ is the countable set of discontinuity points of $F$, $\Delta\Lambda(u)=\Lambda(u)-\Lambda(u^{-})$ is the jump height at $u$ and $\Lambda^c(x)=\Lambda(x)-\sum_{u\leq x, u\in D } \Delta\Lambda(u)$ is the continuous part of $\Lambda$. 
Moreover, under the conditions of Theorem \ref{tgf1} (a), it is immediate using that  $\bar{F}(x)=1-F(x^{-})+f(x)\mu(x)$ (and analogously for $g$)
\begin{equation}f(x)\bar{G}(x)\geq g(x)\bar{F}(x)\ \Leftrightarrow  f(x)(1-G(x^{-}))\geq g(x)(1-F(x^{-})), \quad \hbox{ $\mu$ a.e.} \label{equiv}\end{equation}
We will apply (\ref{surcum}) to the survival lifetimes $X_t$ and $Y_t$. First of all, let us call $\Lambda_1$ and $\Lambda_2$ the cumulative hazard functions of $X$ and $Y$, respectively, and denote by $\Lambda_{1,t}$ and $\Lambda_{2,t}$ the cumulative hazard functions of $X_t$ and $Y_t$, whenever $\bar{F}(t)>0$. It can be easily checked that
\begin{equation}\Lambda_{i,t}(s)=\Lambda_i(t+s)-\Lambda_i(t),\quad i=1,2, \quad s\geq 0. \label{difer}\end{equation}
Let us call $\bar{F}_t$ the survival function of $X_t$, and denote by $D_1$ and $D_2$ the set of discontinuity points of $F$ and $G$. By (\ref{surcum}) and  (\ref{difer}) , we have 
\begin{equation}\bar{F}_t(s)=exp\{-\Lambda_{1,t}^c(s)\}\prod_{\substack{t<u\leq t+s,\\ u\in D_1} }(1-\Delta\Lambda_1(u)), \label {surcu2}\end{equation}
From now on, we will assume that $\bar{F}_t(s)>0$, otherwise $\bar{G}_t(s)\geq\bar{F}_t(s)$ holds trivially.  We have from (\ref{equiv}) that $D_2\bigcap(t,t+s]\subseteq D_1\bigcap(t,t+s]$.  Let us call $g^{c}(x)=g(x)1_{x\not \in D_1}$ and  $f^{c}(x)=f(x)1_{x\not \in D_1}$. Observe that we have by (\ref{equiv})
 \begin{equation}\Lambda_{1,t}^c(s)=\int_{(t,t+s]}\frac{f^{c}(u)}{1-F(u^{-})}d\mu(u)\geq \int_{(t,t+s]}\frac{g^{c}(u)}{1-G(u^{-})}d\mu(u) =\Lambda_{2,t}^c(s) .\label{fact1}\end{equation}
Moreover,
\begin{align}&\prod_{\substack{t<u\leq t+s,\\ u\in D_1}}(1-\Delta\Lambda_1(u))=\prod_{\substack{t<u\leq t+s,\\ u\in D_1}} \left(1-\frac{\mu(u)f(u)}{1-F(u^{-})}\right)\nonumber\\&\leq\prod_{\substack{t<u\leq t+s,\\ u\in D_2}} \left(1-\frac{\mu(u)g(u)}{1-G(u^{-})}\right)=\prod_{\substack{t<u\leq t+s,\\ u\in D_2}}(1-\Delta\Lambda_2(u)) .\label{fact2}\end{align} Therefore, we conclude from (\ref{fact1}), (\ref{fact2}) and (\ref{surcu2}) that $\bar{F}_t(s)\leq \bar{G}_t(s), \ s\geq 0$, thus showing that  $X-t|X> t\leq_{\text{st}} Y-t|Y>t$ and concluding by  (\ref{rhres}) that $X\leq_{\text{hr}}Y$.  This shows the if part. 

For part (b),  let $\bar{F}^{\star}$ ($ \bar{G}^{\star}$) be the survival function of $-X$ ($-Y$). We take into account (cf. \cite[Thm. 1.3.10]{mustst}) that $X\leq_{\text{rh}} Y  \Leftrightarrow -Y\leq_{\text{hr}} -X$. Using part (a), the implication on the right holds if and only if $-X$ and  $-Y$ have a common dominating measure (say $\mu^{\star}$) with respective densities $f^{\star}$ and $g^{\star}$ such that (\ref{equiv}) holds, that is
\begin{equation} g^{\star}(x)\bar{F}^{\star}(x)\geq f^{\star}(x)\bar{G}^{\star}(x)\ \Leftrightarrow  g^{\star}(x)(1-F^{\star}(x^{-}))\geq f^{\star}(x)(1-G^{\star}(x^{-})), \quad \hbox{ $\mu^{\star}$ a.e.}\label{equiv2}\end{equation}
Now consider the measure $\mu$ such that $\mu(B)=\mu^{\star}(-B)$, for all Borel set $B$. It is readily seen that a density for $X$ with respect to $\mu$ is given by $f(x)=f^{\star}(-x)$ and that $F(x)=1-F^{\star}(-x^{-})$. Thus the last inequality in (\ref{equiv2}) holds if and only if 
\[ g(-x)F(-x)\geq f(-x)G(-x),\quad  \mu \hbox{ a.e.} 
\]
this concludes the proof.
 	\end{proof}
 	\begin{rem} If we define the generalized hazard rate and reversed hazard rate of a random variable $X$, having density $f$ with respect to a dominating measure $\mu$ as 
 		\begin{equation}
 			r_X(x):=\frac{f(x)}{\bar{F}(x)}, \quad \text{and}\quad 	\tilde{r}_X(x):=\frac{f(x)}{F(x)} \quad \mu \ a.e.\label{fairev}
 		\end{equation}
 		Theorem \ref{tgf1} shows that
 		\begin{equation}
 			X\leq_{\text{hr}}Y\Leftrightarrow	r_X(x)\geq r_Y(x) \quad \text{and}\quad X\leq_{\text{rh}}Y\Leftrightarrow  \tilde{r}_X(x)\leq \tilde{r}_Y(x) \quad  \mu \ a.e.  \label{frevord}
 		\end{equation}
Observe also that an immediate consequence of the previous result we have that if inequality
 		\[f(x)\bar{G}(x)\geq g(x)\bar{F}(x)\ \hbox{ $\mu$ \  a.e.}\]
 		holds for some dominating measure $\mu$, then it holds for all dominating measures, due to (\ref{thgeh1}) and the limiting argument below (\ref{thgeh1}). 
 		
 		Finally, as mentioned before, in the literature of survival analysis, we can find identity (\ref{equiv}) for nonnegative random variables.  However, it is immediate to extend this identity for real random variables.  We only need to take into account that
 		\[\bar{F}(x)=P(X>x|X>t)P(X>t)=\bar{F}_t(x-t) P(X>t), \quad x\geq t. \]
 		and we have the extension of (\ref{equiv}) for real random variables using (\ref{surcu2}) and letting $t\rightarrow -\infty$.\end{rem}

 \begin{example}
 	Let us consider two non-negative random variables $X$ and $Y$ having piecewise exponential distributions. Specifically, assume there exist $n$ change points $0 < x_1 < \dots < x_n$ (with $x_0 = 0$ and $x_{n+1} = \infty$) such that $X$ has a constant failure rate $\lambda_i$ on the interval $(x_i, x_{i+1})$. Analogously, $Y$ has constant failure rate $\lambda_i^\star$ on $(x_i, x_{i+1})$. 
 	
 	Furthermore, assume that both models can accumulate point masses at $\{0, x_1, x_2, \dots, x_n\}$ with probabilities $p_0, p_1, \dots, p_n$ for $X$, and $p_0^\star, p_1^\star, \dots, p_n^\star$ for $Y$. 
 	
 	In reliability theory, a point mass at $0$ models burn-in defects in a lifetime. Moreover, in survival analysis, a point mass at $x_n$ (such that $p_n = 1 - F(x_n^-)$ together with $\lambda_n = 0$) is commonly used to describe censoring.
 
Moreover, if $X$ represents the duration a client stays with a bank or company, the points $x_i$ represent the instants at which contract conditions change, making it natural to assume a positive probability $p_i$ of a client leaving at $x_i$.
 	
 	A dominating measure for both $X$ and $Y$ is given by Lebesgue measure $\ell$ on $[0,\infty)$  plus counting measure on the discrete set:
 	\[
 	\mu(x) = \ell(x) + \mathbf{1}_{\{0, x_1, \dots, x_n\}}(x), \quad x \in \mathbb{R}.
 	\]
 	Then, applying Theorem~\ref{tgf1} and (\ref{equiv}), $X \leq_{\text{hr}} Y$ if and only if:
 	\[
 	\lambda_i \geq \lambda_i^\star \quad \text{and} \quad r_i := \frac{p_i}{1-F(x_i^{-})} \geq \frac{p_i^\star}{1-G(x_i^{-})} := r_i^\star, \quad i = 0, 1, \dots, n.
 	\]
 	Observe also that in this case, $\Lambda(x_i) - \Lambda(x_i^-) = r_i$. Thus, using (\ref{surcum}), the survival function can be written as:
 	\[
 	\bar{F}(x) = (1 - r_i) e^{-\lambda_i (x - x_i)} \prod_{j=0}^{i-1} (1 - r_j) e^{-\lambda_j (x_{j+1} - x_j)}, \quad x \in [x_i, x_{i+1}), \quad i = 0, \dots, n,
 	\]
 	where by convention $\prod_{j=0}^{-1} \cdot = 1$ when $i = 0$.
 	
 The conditions provided for the hr order do not guarantee the rh order.  Let us consider two models as before where the only point of discontinity is $0$ and $0<p_0<1$. Then, 
 	\begin{equation}X\leq _{rh}Y \quad \Leftrightarrow \quad p_0 \geq p_0^\star\quad \text{and}\quad  \lambda_0^\star\leq \lambda_0\leq \lambda_0^{\star} \frac{p_0(1 - p_0^\star)}{p_0^{\star}(1-p_0)}. \label{miexrh}\end{equation}
Note that if $p_0^{\star}=0$, last condition is simply $\lambda_0^\star\leq \lambda_0$. The first condition is trivial, as
 	\[\frac{p_0^\star}{p_0}=\frac{G(0)}{F(0)}\leq \lim_{x\rightarrow \infty }\frac{G(x)}{F(x)}=1.\]
Now, let us analyze the second condition.  Observe that
\[\tilde{r}_X(x)=\frac{(1-p_0)\lambda_0 e^{-\lambda_0 x}}{p_0+(1-p_0) (1-e^{-\lambda_0 x})}\quad \text{and} \quad \tilde{r}_Y(x)=\frac{(1-p_0^{\star})\lambda_0^{\star} e^{-\lambda_0^{\star} x}}{p_0^{\star}+(1-p_0^{\star}) (1-e^{-\lambda_0^{\star} x})}, \quad x>0. \]
Observe that, in particular, conditions in Theorem~\ref{thgeh1} (b), imply that
\[\lim_{x\rightarrow 0+}\tilde{r}_X(x)\leq \lim_{x\rightarrow 0+}\tilde{r}_Y(x) \ \Leftrightarrow \	\lambda_0 \left(\frac{1 - p_0}{p_0}\right) \leq \lambda_0^\star \left(\frac{1 - p_0^\star}{p_0^\star}\right),\]
 	thus showing the upper bound for $\lambda_0$. For the lower bound, we need that
 	\[d(x)=\frac{1}{\tilde{r}_X(x)} - \frac{1}{\tilde{r}_Y(x)} =\left( \frac{e^{\lambda_{0} x}}{\lambda_{0}(1-p_{0})} - \frac{1}{\lambda_{0}} \right) - \left( \frac{e^{\lambda_{0}^\star x}}{\lambda_{0}^\star(1-p_{0}^\star)} - \frac{1}{\lambda_{0}^\star} \right)\geq 0,\quad x>0.\]
 	Observe that, if $\lambda_{0}< \lambda_{0}^\star$, $\lim_{x\rightarrow \infty}d(x)=-\infty$. Thus necessarily  $\lambda_{0}\geq \lambda_{0}^\star$. In this case, $d(x)$ is increasing, as $d'(x)\geq 0$, thus completing the proof. \label{exexpo}
 \end{example}

	\section{A general proof for the bivariate characterization of the hr and rh orders.} 	 \label{seage}	
 	
Theorem \ref{tgf1} allows us to check in an easy way the celebrated bivariate characterizations of the hr and the rh orders for general random variables, very useful in applied probability problems.  Note that the initial proof provided in \cite{shyabi} for the hr order dealt with the case of absolutely continuous random variables (and skipped some technical details), as well as the  proof given in \cite{kionpo} for the rh order.  Moreover, the proofs given in the book of Muller and Stoyan \cite{mustst} would require an approximation procedure for expectations, and therefore additional integrability conditions. The next proof is general, and it is only based on the characterization given in  Theorem \ref{tgf1} and elementary properties of the hr and rh orders. 
\begin{thm} Let $X$ and $Y$ be two random variables.  We have \begin{enumerate}[(a)]
		\item $X\leq_{\text{hr}} Y$ if and only if
		\[Eh(X^\star,Y^\star)\leq Eh(Y^\star,X^\star), \quad \hbox{for all}\quad h \in {\cal G}_{hr},\]
		where
		\[{\cal G}_{hr}=\{h:\mathbb{R}^2 \longrightarrow\mathbb{R}| \Delta h(x,y):=h(x,y)-h(y,x)\ \hbox{is increasing in $x$ for all $x\geq y$} \}\]
		and $X^\star$ and $Y^\star$ are independent with $X^\star=_{\text{st}}X$ and $Y^\star=_{\text{st}} Y$. 
 \item $X\leq_{\text{rh}} Y$  if and only if
\[Eh(X^\star,Y^\star)\leq Eh(Y^\star,X^\star), \quad \hbox{for all}\quad h \in {\cal G}_{rh},\]
where
\[{\cal G}_{rh}=\{h:\mathbb{R}^2 \longrightarrow\mathbb{R}| \Delta h(x,y)\ \hbox{is increasing in $x$ for all $x\leq y$} \}\]
and $X^\star$ and $Y^\star$ are independent with $X^\star=_{\text{st}} X$ and $Y^\star=_{\text{st}} Y$. 	\end{enumerate}\label{thbiva}\end{thm}
\begin{proof}
Let us show part (a). Let $X\leq_{\text{hr}} Y$, and take a dominating measure $\mu$ such that both random variables are absolutely continuous with respective densities $f$ and $g$. Take $h\in {\cal G}_{hr}$. In analogous way as in the proof of Theorem 1.9.3 in \cite{mustst}, we have
\begin{equation}E\Delta h(X^\star,Y^\star)=\int_y\int_{x\geq y}\Delta h(x,y) (f(x)g(y)-f(y)g(x))d\mu(x)d\mu(y). \label{demo1}\end{equation}
Our aim is to show that the inner integral is less than or equal to $0$ for all $y\in \mathbb{R}$.  Note that if $\bar{F}(y)=0$, as $\Delta h(y,y)=0$,
\[\int_{x\geq y}\Delta h(x,y) f(x)d\mu(x)=0,\]
and the fact is obvious. Therefore, from now on, we will assume that $\bar{F}(y)>0$.
 Note that, as $X^\star\leq_{\text{hr}}Y^\star$, then $X^\star|X^\star > y\leq_{\text{st}}Y^\star|Y^\star>y$ (cf \cite[p .8]{mustst}). Therefore, as $\Delta h(x,y)$ is increasing in $x$ for $x \geq y $, we have (cf \cite[p. 5]{mustst}) 
\[E[\Delta h(X^\star,y)|X^\star>y]\leq E[\Delta h(Y^\star,y)|Y^\star>y]\]
This, the fact that $\Delta h(x,y)\geq 0, \ x\geq y$ and Theorem \ref{tgf1}(a) imply ($y$ a.e. $\mu$) 
\begin{align}&g(y)\int_{x> y}\Delta h(x,y) f(x)d\mu(x)=g(y)\bar{F}(y)\int_{x> y}\Delta h(x,y)\frac{ f(x)}{\bar{F}(y)}d\mu(x)\nonumber\\&=g(y)\bar{F}(y)E[\Delta h(X^\star,y)|X^\star>y]\leq g(y)\bar{F}(y)E[\Delta h(Y^\star,y)|Y^\star>y]\nonumber\\&=\frac{ g(y)}{\bar{G}(y)}\bar{F}(y)\int_{x> y}\Delta h(x,y)g(x)d\mu(x) \leq \frac{ f(y)}{\bar{F}(y)}\bar{F}(y)\int_{x> y}\Delta h(x,y)g(x)d\mu(x)\nonumber\\& =f(y) \int_{x> y}\Delta h(x,y) g(x)d\mu(x). \label{demo2}\end{align}
(\ref{demo2}) and the fact that $\Delta h(y,y)=0$ imply that the inner integral in (\ref{demo1}) is less than or equal to $0$ ($\mu$ a.e. $y$).  Therefore, 
\[E\Delta h(X^\star,Y^\star)\leq 0.\]
The only if case is as in the proof of Theorem 1.9.3 in \cite{mustst}. This concludes part (a).  Part (b) holds in a similar way.
\end{proof}
\section{Preservation of the hazard rate order and the reversed hazard rate order under summation.} \label{sepres}
Our aim in this Section is to use the hr and rh orders to generalize the concept of increasing failure rate and decreasing reversed hazard rate, as well as some classical results related to these concepts. First of all, we recall the main aging concepts we are going to deal with.  For a thorough study of aging concepts in the continuous case, see, for instance Marshall and Olkin \cite{maolli}. For an extensive study of the DRHR property see, for instance,  \cite{senalo}. For the discrete case, we can mention Nair et al. \cite{nareli}. Notice that usually these concepts are defined for non-negative random variables.  However, for the sake of generality, we will use real random variables, as in the previous Sections and as in \cite{shshst}
\begin{defn}
	Let $X$ be a random variable  with cdf $F$ and survival function $\bar {F}$. We say that $X$ is:
	\begin{itemize}
		\item[(a)]  Increasing failure rate or IFR if $\bar {F}$ is log-concave.
		\item[(b)] Decreasing reversed hazard rate or DRHR if $F$ is log-concave.
	\end{itemize}
	\label{dfrelc}\end{defn}
	It is well-known that IFR random variables are absolutely continuous except, possibly at the right end point of its support (if any). Analogously, DRHR random variables are absolutely continuous except, possibly at the left end point of its support (if any). Moreover, if we denote by $f$ the density of $X$ on $Int(supp(X))$, where $Int(supp(X))$ denotes the interior of the support of $X$, we have, using the notations introduced in (\ref{fairev}) that for all $x\in Int(supp(X))$ (see \cite[Theorems 1.B.38 and 1.B.62]{shshst})
\begin{equation}X \ IFR \ \Leftrightarrow  \ r_{X}(x) \ \text{increases and } X \ DRHR \ \Leftrightarrow  \ \tilde{r}_{X}(x) \ \text{decreases}. \end{equation}
Now, we consider some aging properties for discrete random variables. In an analogous way as in the previous Definition, we will make use of the concept of log-concavity (cf. \cite{nasedi}).   Recall that a sequence of non-negative real numbers $\{a_n\}_{n\in \mathbb{Z}}$ is said to be log-concave if it has no internal zeros and
\begin{equation}a_{n+1} \geq (\leq ) a_n^{\frac {1}{2}}a_{n+2}^{\frac {1}{2}},\quad  n\in \mathbb{Z}.  \label{logcod}\end{equation}

\begin{defn}Let $L$ be an integer-valued random variable. $L$ is said to be
	\begin{itemize}
		\item[(a)] Discrete IFR if the sequence  $\{P(L> n)\}_{n\in \mathbb{Z}}$ is log-concave.
		\item[(b)] Discrete DRHR if the sequence  $\{P(L\leq n)\}_{n\in \mathbb{Z}}$ is log-concave.
	\end{itemize} \label{dfreli}
\end{defn}
	It is readily seen that these concepts imply the monotonocity of the hazard rates and reversed hazard rates. In this case, the random variables are absolutely continuous with respect to the counting measure in $\mathbb{Z}$ and $f(n)=P(L=n)$. And for every $n\in supp(L)$,
\begin{align}&L \ IFR \Leftrightarrow r_{L}(n) \ \text{increases}  \Leftrightarrow  \frac{f(n)}{P(L\geq n)} \ \text{increases and } \label{ifreq}\\ &L \ DRHR \Leftrightarrow \tilde{r}_{L}(n) \ \text{decreases}. \label{drhreq}\end{align}
	The second equivalence in (\ref{ifreq}) is readily seen using (\ref{equiv}).  See also \cite[Lemma 2]{balogc}. The first equivalence in both parts is immediate (see, for instance \cite[pp. 181 and 225]{nareli}). Note that the discrete IFR property is usually defined by means of the second equivalence in (\ref{ifreq}), because it has a more intuitive interpretation.  However, with the use of $r_L(n)$ we unify our proofs in the continuous and discrete cases.
	 
It is well-known that an absolutely continuous random variable $X$ is IFR if and only if $X\leq_{\text{hr}}X+y$ for all $y>0$ (cf. \cite[p. 56]{maolli}). This characterization shows the tight relation between failure rates and hazard rate orders.   It is also known (cf. \cite{blsath}) that if X is an absolutely continuous IFR random variable, then $X\leq_{\text{hr}}X+Y$ for every absolutely continuous non-negative random variable $Y$ (not necessarily IFR).  Our aim in Theorem \ref{thconv} is to generalize the class of distributions verifying this last property. In this generalization, the bivariate characterization of the hazard rate order will play an important role. First of all, we begin with a definition.
\begin{defn} Let $X$ be a random variable.  Let us define the sets
	\begin{align}A_{hr}&:=\{y\geq 0 \text{ such that } X\leq_{\text{hr}} X+y\},\\
	A_{rh}&:=\{y\geq 0 \text{ such that } X\leq_{\text{rh}} X+y\}\end{align}
	We will say that $X$ is $A_{hr}$ monotone in the hazard rate order and $A_{rh}$ monotone in the reversed hazard rate order.\end{defn}
	\begin{rem}	Obviously, $A_{hr}$ or $A_{rh}$ can be the set $\{0\}$. As said before, for an absolutely continuous IFR random variable $X$, $A_{hr}=\mathbb{R}_+$.  This property can be extended to any IFR random variable, due to the log-concavity property of $\bar{F}$. Similarly, we can see that a discrete IFR random variable with unbounded upper support  is $\mathbb{N}$ monotone in the hazard rate order. If it is non-negative and the last endpoint of its support is $n$, $A_{hr}=\{0,1,\dots,n\}\bigcup  (n,\infty)$. Similar considerations can be made for DRHR random variables and the reversed hazard rate order. \label{reset}\end{rem}

	Using the bivariate caracterization of the hazard rate order, we have the following 
	\begin{thm} 
		\begin{enumerate}[(a)]
			\item If $X$ is $A_{hr}$ monotone in the hazard rate order, being $A_{hr}\neq \{0\}$, then $X\leq_{\text{hr}}X+Y$, for any random variable $Y$ independent of $X$ having support in $A_{hr}$. 
				\item If $X$ is $A_{rh}$ monotone in the reversed hazard rate order, being $A_{rh}\neq \{0\}$, then $X\leq_{\text{rh}}X+Y$, for any random variable $Y$ independent of $X$ having support in $A_{rh}$.  \end{enumerate}\label{thconv}\end{thm}
			\begin{proof} To show part (a), let $X_1$ and $X_2$ be two independent copies of $X$, also independent of $Y$. Then, by Theorem \ref{thbiva} (a), we have
				\[E\Delta h(X_1,X_2+y)\leq 0, \quad \hbox{for all}\quad h \in {\cal G}_{hr},\quad y\in A_{hr}.\]
				Let $F_{Y}$ be the distribution function of $Y$.  Using the previous inequality we have for any $h \in {\cal G}_{hr}$
				\[E\Delta h(X_1,X_2+Y)=\int_{A_{hr}} E\Delta h(X_1,X_2+y)dF_{Y}(y)\leq 0\]
			thus concluding, using again Theorem \ref{thbiva} (a) that $X\leq_{\text{hr}}X+Y$. The proof of part (b) is analogous. \end{proof}
					
			As an immediate consequence of Theorem \ref{thconv} we have the following Corollary. In part (a), we generalize the result given in \cite{blsath} to general random variables $Y$.  Moreover, in part (b), we provide the discrete version of this result.
			\begin{cor}
				\begin{enumerate}[(a)]
					\item An IFR (respectively, DRHR) random variable $X$ verifies that $X\leq_{\text{hr}}X+Y$ (respectively  $X\leq_{\text{rh}}X+Y$), for any non-negative random variable $Y$ independent of $X$.
					\item A discrete random variable $X$, taking values on $\mathbb{N}$, having discrete IFR (respectively, DRHR) verifies that $X\leq_{\text{hr}}X+Y$ (respectively  $X\leq_{\text{rh}}X+Y$), for any random variable $Y$ independent of $X$ and taking values on $\mathbb{N}$.
				\end{enumerate}
			\end{cor}

\begin{example} Let us consider a $\mathbb{Z}$-valued random variable having discrete IFR. Obviously, $2X$ does not have discrete IFR.  However it can be readily seen that $\{2i, \ i\in \mathbb{N}\}\subseteq A_{hr}$. Thus, applying Theorem \ref{thconv} (a), $2X\leq_{\text{hr}} 2X+Y$ if $Y$ is independent of $X$ and takes values on $\{2i, \ i\in \mathbb{N}\}$. In an analogous way we would have $X\leq_{\text{hr}} X+Y$ if $X$ is a discrete random variable whose hazard rate is increasing on the even numbers, and also increasing on the odd numbers. \end{example}
It is obvious that if $X$ is $A_{hr}$ monotone in the hazard rate order, then $X+Y$ does not, necessarily, inherit this property. Consider the case $X=0$ (IFR) and $Y\neq 0$ having a decreasing failure rate.  Then, $X+Y=Y\neq{0}$ has a decreasing failure rate, and therefore it is $\{0\}$ monotone in the hazard rate order. Thus, we will need additional conditions on $Y$ to extend Theorem \ref{thconv} to convolutions of higher order.

First of all, observe that the sets $A_{hr}$ and $A_{rh}$ verify the following property
\begin{equation} x, y \in A_{hr}\Rightarrow x+y \in A_{hr}\quad \hbox{and}\quad x, y \in A_{rh}\Rightarrow x+y \in A_{rh}.\label{conrh}\end{equation}
This is due to the fact that $X\leq_{\text{hr}} X+x \Rightarrow X+y\leq_{\text{hr}} X+x+y$, due to the preservation of the hr order under increasing transforms (see \cite[Thm. 1.3.7.]{mustst}).  Thus, $X\leq_{\text{hr}} X+y \leq_{\text{hr}} X+x+y$ and the transitivity of the hr order implies $X\leq_{\text{hr}} X+x+y$, which proves (\ref{conrh}) for the hr order.  The proof for the rh order is similar. Moreover, a simple application of (\ref{conrh}), together with the closure of the hr and rh orders under weak convergence (see \cite[Thm. 1.3.5.]{mustst}), provides us that if $X$ and $Y$ are independent,
\begin{equation}supp(X)\subseteq A_{hr},\  supp(Y)\subseteq  A_{hr} \Rightarrow supp(X+Y)=\overline{supp(X)+supp(Y)}\subseteq  A_{hr} \label{conrs}.\end{equation}
and similar considerations can be made for the rh order.
However, observe that $x, y \in A_{hr},\ x<y$ does not necessarily imply $y-x \in A_{hr}$. In order to obtain our generalization we will need this property.  Thus we present the following definition
\begin{defn}Let $X$ be a random variable having support on $\overline{G}$, the closure of an additive group $G\subseteq\mathbb{R}$. We will say that $X$ is $G$-IFR (respectively $G$-DRHR) if $X+x\leq_{\text{hr}}X+y$ (respectively $X+x\leq_{\text{rh}}X+y$) for all $x, y$ in $G$ such that $0\leq x\leq y$. \label{dfgclas} \end{defn}
Absolutely continuous IFR random variables are a particular case of this definition, taking $G=\mathbb{R}$.  (cf. \cite[Thm. 1.B.38]{shshst}).  As in Remark \ref{reset}, this
 property can be extended to any IFR random variable, due to the log-concavity property of $\bar{F}$.  If we take $G=\mathbb{Z}$, we obtain the class of discrete increasing failure rates. Moreover, $G$ can be a lattice of points, or a dense set in $\mathbb{R}$ such as the rational numbers $\mathbb{Q}$. 

In Definition \ref{dfgclas} we included random variables having support in $\overline{G}$ in order to include convolutions into the same class. Moreover, we only assumed increasingness in the hazard rate (reversed hazard rate) in $G\cap [0,\infty)$.  However, using the closure of the hazard rate order under increasing transforms and weak convergence we can extend this result to  $\overline{G}$.

\begin{lem} Let $G\subseteq\mathbb{R}$ be an additive group and let $X$ be a $G$-IFR  random variable  (respectively $G$-DRHR).  Then $X+x\leq_{\text{hr}}X+y$ (respectively $X+x\leq_{\text{rh}}X+y$) for all $x, y$ in $\overline{G}$ such that $x\leq y$.\label{lextgf}\end{lem}
\begin{proof} First of all, let us assume that $x,y\in G$.  Observe that if $x\leq y$, then $y-x\geq 0$. As $G$ is an additive group $y-x\in G$. Thus, the $G$-IFR property of $X$ implies $X\leq_{\text{hr}}X+y-x$. As the hazard rate order is preserved under increasing transforms, we add $x$ in both sides of the inequality, thus obtaining that $X+x\leq_{\text{hr}}X+y$. If either $x\not \in G$ or $y\not \in G$, we would use an approximation procedure and the closure of the hazard rate order under weak convergence (for instance, if $x\not \in G$, we take a sequence of points $x_n\leq x$ in $G$ to approximate $x$ and use $X+x_n$ to approximate $X+x$).  The proof for the reversed hazard rate is similar.\end{proof}
\begin{rem} For absolutely continuous random variables, it is well-known (cf. \cite[Thm. 1.B.38]{shshst}) that 
	\[X+x\leq_{\text{hr}}X+y \hbox{ for all $x, y$ in $\mathbb{R}$ such that $x\leq y$} \Leftrightarrow X \hbox{ is IFR}.\]
The previous Lemma shows us that there are also discrete random variables satisfying the left hand side part of the previous assertion.  For instance $\mathbb{Q}$-valued random variables being $\mathbb{Q}$-IFR, as $\overline{\mathbb{Q}}=\mathbb{R}.$
On the other hand, observe that Lemma \ref{lextgf} establishes equivalence classes for the $G$-IFR or $G$-DRHR properties, as provides the same result for  additive groups $G_1$ and $G_2$ such that $\overline{G}_1=\overline{G}_2$. Thus, from now on, we will speak of $X$ being in the $\overline{G}$ IFR or DRHR class, when Definition \ref{dfgclas} holds.

Moreover, according to the classification of additive subgroups in $\mathbb{R}$ (cf. \cite[Lemma 2.2.]{abmato}), we have only two possible situations for $G$:  
\begin{enumerate}[1)]
	\item $G$ is a dense set in $\mathbb{R}$. In this case, we have $\overline{G}=\mathbb{R}$.
	\item $G$ is a lattice of points of period $d$.  In this case $\overline{G}=G$.
	\end{enumerate} 
Thus the main advantage of the following results, is that we unify properties in the discrete and continuous case, and also that the $\mathbb{R}$-IFR or $\mathbb{R}$-DRHR classes are not only composed of classical IFR or DRHR random variables.  \end{rem}
The next result shows us the closure of the $\overline{G}$-IFR class under convolution. For a proof for IFR random variables, see for instance \cite[Thm. C.4.]{maolli}. Our proof relies on a classical result of preservation of the $hr$ order under mixtures (see \cite[Theorem 1.B.14]{shshst}).

	\begin{thm} Let $X_1,\dots,X_n$ be independent random variables.
	\begin{enumerate}[(a)] \item If $X_1,\dots,X_n$ are in the $\overline{G}$-IFR class, then $X_1+\dots+X_n$ is in the $\overline{G}$-IFR class.
\item If $X_1,\dots,X_n$ are in the $\overline{G}$-DRHR class, then $X_1+\dots+X_n$ is in the $\overline{G}$-DRHR class.     \end{enumerate}\label{thclco}\end{thm}
\begin{proof} First of all, note that using a similar procedure as in (\ref{conrs}), we deduce that the support of $X_1+\dots+X_{k}$ is in $\overline{G}$ for $k=2,\dots,n$. To show part (a), let us proceed by induction.  The case $n=1$ is obvious by hypothesis.  Now, assume that $X_1+\dots+X_{n-1}$ is in the $\overline{G}$-IFR class. Then, we have
	\begin{equation}X_1+\dots+X_{n-1}+x\leq_{\text{hr}} X_1+\dots+X_{n-1}+y \hbox{ for $x, y$ in $G$ such that $0\leq x\leq y$}.\label{mixco1}\end{equation}
	Moreover, as $X_n$ is in the $\overline{G}$-IFR class, using Lemma \ref{lextgf}, we have
	\begin{equation}X_n+x\leq_{\text{hr}}X_n+y,\quad \hbox{for all $x,y \in \overline{G}$}. \label{mixco2} \end{equation}
On the other hand we can prove, using a similar procedure as in (\ref{conrh}) that the supports of both $X_1+\dots+X_{n-1}+x$ and $X_1+\dots+X_{n-1}+y$ in (\ref{mixco1}) are both in $\overline{G}$.  Thus, by (\ref{mixco1}) and (\ref{mixco2}) and applying \cite[Theorem 1.B.14]{shshst}, we have
\[X_1+\dots+X_{n}+x\leq_{\text{hr}} X_1+\dots+X_{n}+y \hbox{ for $x, y$ in $G$ such that $0\leq x\leq y$},\]
thus completing part (a). The proof of part (b) is similar.\end{proof}

A classical result in stochastic orders is that if $(X_i,Y_i)$, $i=1,n$ are independent pairs of random variables such that $X_i\leq_{\text{hr}}Y_i$, $i=1,\dots, n$, and $X_1,\dots,X_n,Y_1,\dots, Y_n$ are all IFR (except, possibly one $X_j$ and one $Y_l$, with $j\neq l$), then $X_1+\dots+X_n\leq_{\text{hr}}Y_1+\dots+Y_n$. See, for instance, \cite[Thm. 1.B.4]{shshst}. We are going to extend this result to random variables in the $\overline{G}$-IFR class, using a similar technique as in the proof of Theorem \ref{thclco}. 

	\begin{thm} Let $(X_i,Y_i)$, $i=1,2,\dots,n$ be independent pairs of 
		random variables. 
	\begin{enumerate}[(a)] \item If $X_1,\dots,X_n,Y_1,\dots, Y_n$ are in the $\overline{G}$-IFR class (except, possibly, one $X_j$ and one $Y_l$ with $j\neq l$), and $X_i\leq_{\text{hr}} Y_i$, $i=1,2,\dots,n$, then $X_1+\dots+X_{n}\leq_{\text{hr}}  Y_1+\dots+Y_{n}$.
		\item If $X_1,\dots,X_n,Y_1,\dots, Y_n$ are in the $\overline{G}$-DRHR class (except, possibly, one $X_j$ and one $Y_l$ with $j\neq l$),
		 and $X_i\leq_{\text{rh}} Y_i$, $i=1,2,\dots,n$, then $X_1+\dots+X_{n}\leq_{\text{rh}}  Y_1+\dots+Y_{n}$.     \end{enumerate}\label{thpcon}\end{thm}
\begin{proof} To show 
	part (a), assume that in the first pair $Y_1$ is in $\overline{G}$-IFR class, in the last pair $X_n$ is in $\overline{G}$-IFR class, and $X_2,\dots,X_{n-1},Y_2,\dots, Y_{n-1}$ are all in $\overline{G}$-IFR class (if not, we can reorganize the pairs in order to get this). Now, let us proceed by induction on $k\leq n$.  The case $k=1$ is obvious by hypothesis. Assume that 
	\begin{equation}X_1+\dots+X_{k-1}\leq_{\text{hr}} Y_1+\dots+Y_{k-1},\quad k\leq n.\label{mixco3}\end{equation} Observe that, as $X_k$ is in $\overline{G}$-IFR class, (\ref{mixco2}) holds.
	 Thus, applying (\ref{mixco2}), (\ref{mixco3}) and \cite[Theorem 1.B.14]{shshst}, we obtain that
	\begin{equation}X_1+\dots+X_{k-1}+X_k\leq_{\text{hr}}Y_1+\dots+Y_{k-1}+X_k\label{mixco4}\end{equation}
Moreover, using Theorem \ref{thclco} (a) $Y_1+\dots+Y_{k-1}$ is in $\overline{G}$-IFR class, and then (\ref{mixco2}) holds, that is
	\begin{equation}Y_1+\dots+Y_{k-1}+x\leq_{\text{hr}}Y_1+\dots+Y_{k-1}+y,\quad \hbox{for all $x,y \in \overline{G}$}. \label{mixco5} \end{equation}
Applying now (\ref{mixco5}), the fact that $X_k\leq_{\text{hr}} Y_k$ and \cite[Thm. 1.B.14]{shshst}, we obtain that
	\begin{equation}Y_1+\dots+Y_{k-1}+X_k\leq_{\text{hr}} Y_1+\dots+Y_{k-1}+Y_k.\label{mixco6}\end{equation}
	The conclusion follows by (\ref{mixco4}), (\ref{mixco6}) and the transitivity of the hr order. The proof of part (b) is similar.
\end{proof}
As an immediate application of the previous result, we have the following
	\begin{cor} Let $(X_i,Y_i)$, $i=1,2,\dots,n$ be independent pairs of 
	random variables.  Let $Y_{n+1},\dots Y_m$, $m>n$, be non-negative independent random variables, also independent of the previous pairs. 
	\begin{enumerate}[(a)] \item If $X_1,\dots,X_n,Y_1,\dots, Y_m$ are in the $\overline{G}$-IFR class (except, possibly, one $X_j$ and one $Y_l$ with $j\neq l$) and $X_i\leq_{\text{hr}} Y_i$, $i=1,2,\dots,n$, then $X_1+\dots+X_{n}\leq_{\text{hr}}  Y_1+\dots+Y_{m}$.
		\item If $X_1,\dots,X_n,Y_1,\dots, Y_m$ are in the $\overline{G}$-DRHR class (except, possibly, one $X_j$ and one $Y_l$ with $j\neq l$), and $X_i\leq_{\text{rh}} Y_i$, $i=1,2,\dots,n$, then $X_1+\dots+X_{n}\leq_{\text{rh}}  Y_1+\dots+Y_{m}$.     \end{enumerate}\label{cosum}\end{cor}
	\begin{proof}We only need to apply Theorem \ref{thpcon}, adding to the initial $n$ pairs, $(0,Y_{n+1}),\dots,(0,Y_{m})$. Observe that $0$ is in the $\overline{G}$-IFR class and in the $\overline{G}$-DRHR class for all $G$.\end{proof}
	
	This result is shown for absolutely continuous IFR random variables in \cite[Lemma 3.10]{shyabi}.

\section{Stochastic comparisons of random sums in the hr and the rh orders}		\label{sestoc}

We now deal with the stochastic comparisons of random sums. In Shaked and Shanthikumar \cite[Thm. 1.B.7]{shshst} it was shown that $\sum_{i=1}^N X_i\leq_{\text{hr}}\sum_{i=1}^M X_i$ if $\seq{X}$ are independent non-negative IFR random variables, $M$ and $N$ are non-negative integer valued random variables independent of the previous sequence and $N\leq_{\text{hr}}M$. The condition for the $X_i$ can also be generalized in a simple way using the concept of monotone convolution survival ratio. This class of functions, in fact, characterizes the preservation of the hr order under random summation.
\begin{defn}
	Let $\seq{X}$ be a sequence of independent random variables. We say that  $\seq{X}$ has
	\begin{enumerate}
		\item  Monotone convolution ratio or MCR if $X_1+\dots +X_n\leq_{\text{rh}}X_1+\dots X_{n+1}$ for each $n=1,2,\dots$
		\item   Monotone convolution survival ratio or MCSR if $X_1+\dots +X_n\leq_{\text{hr}}X_1+\dots X_{n+1}$ for each $n=1,2,\dots$
	\end{enumerate}
	\label{dfrelc}\end{defn}
\begin{rem}
	As an immediate consequence of Corollary \ref{cosum} (a), we see that non-negative random variables in the $\overline{G}$-IFR class are MCSR. 	This includes the IFR class of nonnegative random variables. Moreover, Corollary \ref{cosum} (b) implies that non-negative random variables in the $\overline{G}$-DRHR class are MCR, thus including the decreasing reversed hazard rate case. Furthermore, if we consider $G=\mathbb{Z}$, we see that many discrete distribution verify both properties.  For instance,  Binomial, Poisson or negative binomial random variables (see \cite[p. 198]{adinve} for a detailed discussion). 
		We are in a position to enunciate the following. 
\end{rem}
	\begin{prop} 	Let $\seq{X}$ be a sequence of independent non-negative random variables. Let $N$ and $M$ be independent non-negative random variables, also independent of $\seq{X}$.  Then it follows
		\begin{enumerate}[(a)]
			\item $\sum_{i=1}^N X_i\leq_{\text{hr}}\sum_{i=1}^M X_i$ for all $N\leq_{\text{hr}} M$ if and only if $\seq{X}$ is MCSR.  
				\item $\sum_{i=1}^N X_i\leq_{\text{rh}}\sum_{i=1}^M X_i$ for all $N\leq_{\text{rh}} M$ if and only if $\seq{X}$ is MCR.  
		\end{enumerate}\label{promas}\end{prop}
\begin{proof} Let us prove part (a). Let us assume that  $\seq{X}$ is non-negative and MCSR (and therefore, $0\leq_{\text{hr}} X_1$). This property, together with $N\leq_{\text{hr}}M$, implies that $\sum_{i=1}^M X_i\leq_{\text{hr}}\sum_{i=1}^N X_i$, as follows from \cite[Thm.1.B.14]{shshst}. The only if part is clear, as $n\leq_{\text{hr}}n+1$.  This concludes the proof of part (a).  Part (b) is shown in a similar way.
\end{proof}
It seems natural to investigate if Proposition \ref{promas} can be extended to random sums, with different summands in each random sum.  However, additional conditions are needed.  For instance, in the last part of Example \ref{exexpo} we dealt with an exponential random variable with a point mass at $0$ of probability $p$. This corresponds to the random sum $\sum_{i=1}^N X$, being $N$ a Bernoulli random variable and $X$ an exponential random variable with parameter $\lambda$. Let $Y$ be an exponential random variable with parameter $\lambda^{\star}$. It is clear that $X\leq_{hr}Y$. From (\ref{miexrh}) we can deduce that
\begin{equation}
\sum_{i=1}^N X\leq_{\text{rh}} \sum_{i=1}^N Y \Leftrightarrow \lambda=\lambda^{\star}.\label{couexp}
\end{equation}
Moreover reverting the sign in the previous formula we have
\[\sum_{i=1}^N (-Y)\leq_{\text{hr}} \sum_{i=1}^N (-X) \Leftrightarrow \lambda=\lambda^{\star}.\] 
This shows that the hazard rate ordering of the summands is not a sufficient condition to guarantee the hazard rate ordering of the random sum, even if the random summation index does not include positive mass at $0$, as
\[\sum_{i=1}^N X=\sum_{i=1}^{N+1} X_i,\quad \hbox{with} \quad X_1=0 \quad \hbox{and} \quad X_2=X\]

Shanthikumar and Yao \cite[Tm.3.11]{shyabi} provided sufficient conditions for the preservation of the hr order under random sums for absolutely continuous summands.  We revisit their result, simplifying some assumptions and generalizing them to absolute continuity with respect to a dominating measure. Also, we include the possibility for the random summation index to take strictly positive values at $0$.  
	\begin{thm} Let $\seq{X}$ and $\seq{Y}$ be two sequences of independent random variables, being absolutely continuous with respect to a dominating measure $\mu$, and having respective densities $f_{X_i }$ and $f_{Y_i}$. Let $M$ and $N$ be two random variables defined on $\mathbb{N}$, such that $N$ is independent of $\seq{X}$ and $M$ is independent of $\seq{Y}$. Let $T_n=X_1+\dots+X_n$ and $S_n=Y_1+\dots+Y_n$ ($S_0=T_0=0$.). Assume that $N\leq_{\text{lr}}M$.
	\begin{enumerate}[(a)] \item If for all $n\in Supp(N)$ and $m\in Supp(M)$, $T_n\leq_{\text{hr}} S_m$, $n\leq m$ and the following condition holds, (for each $n< m$, $\mu$ a.e.):
		\begin{equation}
\bar{F}_{S_m}(x) f_{T_n}(x) - f_{S_m}(x) \bar{F}_{T_n}(x) \geq \bar{F}_{T_m}(x) f_{S_n}(x)- f_{T_m}(x) \bar{F}_{S_n}(x)\label{chorhr}\end{equation}
	then $T_N\leq_{\text{hr}}S_M$
		\item If for all $n\in Supp(N)$ and $m\in Supp(M)$ $T_n\leq_{\text{rh}} S_m$, for $n\leq m$,  and the following condition holds  (for $n<m$, $\mu$ a.e.):
		\begin{equation}
			F_{S_m}(x) f_{T_n}(x) - f_{S_m}(x)F_{T_n}(x) \leq F_{T_m}(x) f_{S_n}(x) - f_{T_m}(x) F_{S_n}(x),\label{chorrh}\end{equation} then $T_N\leq_{\text{rh}}S_M$ \end{enumerate}\label{thgera}\end{thm} 
		\begin{proof}To show part (a), let $p^{\star}_m = P(M = m)$ and $p_n = P(N = n)$ ($m, n =0, 1, 2, \dots$). Let us assume with no loss of generality that $\mu(\{0\})=1$ (otherwise we modify the dominating measure and the densities, if needed). Note that the respective densities of $T_N$ and $S_M$ with respect to $\mu$ are given by
			\begin{equation}
				\sum_{n=0}^\infty f_{T_n}(x)p_n,\quad \hbox{and}\quad 	\sum_{n=0}^\infty f_{S_n}(x)p^{\star}_n\label{densis}
			\end{equation}
			where $T_0=S_0=0$, and $f_{T_0}=f_{S_0}=1_{\{0\}}(x)$, $x\in \mathbb{R}$. Observe that due to Theorem \ref{tgf1} (a), we need to check that 
			\begin{equation}
			\sum_{n=0}^\infty f_{S_n}(x)p^{\star}_n\sum_{n=0}^\infty 	\bar{F}_{T_n}(x)p_n\leq\sum_{n=0}^\infty f_{T_n}(x)p_n\sum_{n=0}^\infty	\bar{F}_{S_n}(x)p^{\star}_n \label{desi0}
		\end{equation}
			Let us fix $m> n$ and collect all the terms in $m$ and $n$. The previous inequality will hold if 
		\begin{equation}	\left(f_{S_n}(x)\bar{F}_{T_m}(x)-f_{T_m}(x)\bar{F}_{S_n}(x)\right)p^{\star}_np_m+ \left(f_{S_m}(x)\bar{F}_{T_n}(x)-f_{T_n}(x)\bar{F}_{S_m}(x)\right)p^{\star}_mp_n\leq 0 \label{desi1}
		\end{equation} and, collecting all the terms in $n$, we need  for all $n\in Supp(N)\bigcap Supp(M)$
		\begin{equation}
		\left(f_{S_n}(x)\bar{F}_{T_n}(x)-f_{T_n}(x)\bar{F}_{S_n}(x)\right)p^{\star}_np_n\leq 0\label{desi2}
		\end{equation}
		As $N\leq_{\text{lr}}M$, then $p^{\star}_np_m\leq p^{\star}_mp_n, \ m\geq n$. As both inequalities are trivial if $p^{\star}_mp_n=0$, let us consider from now on the rest of the cases. As $T_n\leq_{\text{hr}}S_n$, (\ref{desi2}) holds for $n\in Supp(N)\bigcap Supp(M)$. Moreover, as $T_n\leq_{\text{hr}}S_m$ for $n\in Supp(N)$ and $m\in Supp(M)$, the second term in (\ref{desi1}) is non-positive.   Therefore, we can bound the expression in 
			(\ref{desi1}) by
			\[	\left(f_{S_n}(x)\bar{F}_{T_m}(x)-f_{T_m}(x)\bar{F}_{S_n}(x)+f_{S_m}(x)\bar{F}_{T_n}(x)-f_{T_n}(x)\bar{F}_{S_m}(x)\right)p^{\star}_np_m\leq 0,\]
		where in the last inequality we have used (\ref{chorhr}). This shows (\ref{desi1}) and completes the proof of part (a).  The proof of part (b) is similar. \end{proof}
		Observe that, in \cite[Thm. 3.2]{shyabi}, a more complex condition than the one given in (\ref{chorhr}) was required. Our condition only requires to evaluate condition in \cite[Thm. 3.2]{shyabi}  on the diagonal $y=x$. Moreover, we use general densities so that we can work with discrete or mixed random variables, and include the possibility of a point mass at $0$ in the random summation indices.  
		Finally, under the non-negativity and MCSR assumption of one of the sequences, we can replace $N\leq_{\text{lr}}M$ by $N\leq_{\text{hr}}M$, as it is shown in the following. 
		\begin{cor} Let $\seq{X}$ and $\seq{Y}$ be two sequences of independent random variables, being absolutely continuous with respect to a dominating measure $\mu$, and having respective densities $f_{X_i }$ and $f_{Y_i}$. Let $M$ and $N$ be two random variables defined on $\mathbb{N}$, such that $N$ is independent of $\seq{X}$ and $M$ is independent of $\seq{Y}$. Let $T_n=X_1+\dots+X_n$ and $S_n=Y_1+\dots+Y_n$. We have the following
			\begin{enumerate}[(a)] \item If the conditions in Theorem \ref{thgera} (a) for $\seq{X}$ and $\seq{Y}$ are satisfied and, in addition, either  $\seq{X}$ or $\seq{Y}$ are non-negative and MCSR, 
				then $T_N\leq_{\text{hr}}S_M$, when $N\leq_{\text{hr}}M$. 
				\item If the conditions in Theorem \ref{thgera} (b) are satisfied and, in addition, either  $\seq{X}$ or $\seq{Y}$ are non-negative and MCR, then $T_N\leq_{\text{rh}}S_M$ when $N\leq_{\text{rh}}M$. \end{enumerate}
			\label{coge}\end{cor} 
\begin{proof}Let us prove part (a).  Let us assume that $\seq{Y}$ is non-negative and MCSR.  Then, $N\leq_{\text{lr}}N$, and therefore, $T_N\leq_{\text{hr}}S_N$. Moreover, applying Proposition \ref{promas} (a) $S_N\leq_{\text{hr}} S_M$. The conclusion follows applying the transitivity property of the hr order. If  $\seq{X}$ is non-negative and MCSR, we similarly would have the chain of inequalities $T_N\leq_{\text{hr}}T_M\leq_{\text{hr}} S_M$. The proof of part (b) is similar.   \end{proof}
\begin{rem} It is interesting to examine the case of random sums where  $p_0 = P(N = 0)>0$. Observe that in this case, conditions in Theorem \ref{thgera} (a) imply that $0\leq_{\text{hr}}S_m$, $m\in Supp(M)$, thus implying that $S_m$ is non-negative for $m\in Supp(M)$.  Under this condition and $T_m\leq_{\text{hr}}S_m$, (\ref{chorhr}) would be automatically satisfied, for any $m\in  Supp(M)$ and $n=0$. Details are omitted. The case of the rh order would be more complicated, because, in addition to $0\leq_{\text{rh}}S_m$, condition (\ref{chorrh}) for $m\in  Supp(M)$ and $n=0$ would be
		\[	F_{S_m}(x) f_{T_0}(x) - f_{S_m}(x)F_{T_0}(x) \leq F_{T_m}(x) f_{S_0}(x) - f_{T_m}(x) F_{S_0}(x)\]
and we would need	
	\[f_{T_m}(x)\leq f_{S_m}(x),\quad x>0, \quad m\in Supp(M).\]
	If  both $T_m$ and $S_m$ are strictly positive, this inequality would imply that $T_m=_{\text{st}}S_m$.  A specific example of this strong condition with exponential distributions has been given in (\ref{couexp}).\end{rem}
\section*{Acknowledgments}
The author was supported by Gobierno de Aragón Research Project E48\_23R, and by the Spanish Research Project PID2024-155364NB-I00 (MINECO/FEDER).


\begin{thebibliography}{9}
			 \bibitem{abmato} Abels, H. and Manoussos, A., 2012.  Topological generators of abelian Lie groups and hypercyclic finitely generated abelian semigroups of matrices. {\em Adv. Mathemat.}, {\bf 229,} 1862--1872.
		 \bibitem{adinve} Adrián, D., Bad\'{\i}a, F.G. and Sang\"{u}esa, C., 2026.  Inventory models with nonlinear costs and random lead
		 times: a general framework and some generalizations using
		 reliability techniques. {\em TOP}, {\bf 34,} 187--206.
		  \bibitem{anstat}  Andersen, P.K., Borgan, O., Gill, R.H. and Keiding, N. 1993. {\it Life Statistical Models Based on Counting Processes}, Springer: New York.
	 \bibitem{balogc} Bad\'{\i}a, F.G., Federgruen, A. and Sang\"{u}esa, C., 2021.  Log-concavity of compound distributions with applications in operational and actuarial models, {\em Probab. Eng. Inform. Sci.}, {\bf 35,} 210--235.
	 \bibitem{beanin} Belzunce, F., Martínez-Riquelme, C., and Mulero, J. 2015. {\it An Introduction to Stochastic Orders}. Academic Press.
  \bibitem{biprob} Billingsley, P., 1995. {\it Probability and Measure}, 3rd Edition, Wiley: New York.
  \bibitem{blsath} Block, H.W., and Savits, T.H., 2015. The failure rate of a convolution dominates the failure rate of any IFR component., {\it Stat. Probab. Lett.}, {\bf 107}, 142--144.
 	 \bibitem{jaknin} Janssen, A. and Knoch, A. 2016.  Information bounds for nonparametric estimators of L-functionals and survival functionals under censored data, {\em Metrika}, {\bf 79,} 195--220.
    \bibitem{kionpo} Kijima, M., and Ohnishi, M., 1996. Portfolio selection problems via the bivariate characterization of stochastic dominance relations, {\it Math. Finance}, {\bf 6}, 237--277.
     \bibitem{maolli} Marshall, A.W. and Olkin, I., 2007. {\it Life Distributions}, Springer: New York.
   \bibitem{mustst} M\"uller, A. and Stoyan, D., 2002. {\it Comparison Methods for
  	Stochastic Models and Risks}, Wiley: Chichester, UK.
  	   \bibitem{nareli} Nair, N.U., Sankaran, P.G. and Balakrishnan, N., 2008. {\it Reliability Modelling and Analysis in Discrete Time}, Academic Press: Chichester, UK.
  	   \bibitem{nasedi} Nanda, A.K. and Sengupta, D., 2005.  Discrete life distributions with Decreasing Reversed Hazard Rate, {\it Sankhya}, {\bf 67,} 106--125.
  	\bibitem{roinve} Rosling, K., 2002.  Inventory cost rate functions with nonlinear shortage costs. {\it Oper. Res.}, {\bf 50,} 797--809.
  	\bibitem{senalo} Sengupta, D. and Nanda, A.K., 1999.  Log-concave and concave distributions in reliability. {\it Nav. Res. Log.}, {\bf 46}, 419-433.
  	\bibitem{shshst} Shaked, M. and Shanthikumar, J.G., 2007. {\it Stochastic Orders}. Springer: New York.
  	  \bibitem{shyabi} Shanthikumar, J.G., and Yao, D.D., 1991. Bivariate characterizations of some stochastic orders, {\it Adv. Appl. Probab.}, {\bf 23}, 642-659.
\end{thebibliography}
\end{document}